\documentclass[11pt]{article}

\usepackage[a4paper,margin=1in]{geometry}
\usepackage{amsmath,amssymb,amsthm,mathtools}
\usepackage{microtype}
\usepackage{xcolor}
\usepackage{enumitem}
\usepackage{hyperref}
\usepackage{aliascnt}
\usepackage{cleveref}

\crefname{theorem}{Theorem}{Theorems}
\Crefname{theorem}{Theorem}{Theorems}
\crefname{lemma}{Lemma}{Lemmas}
\Crefname{lemma}{Lemma}{Lemmas}
\crefname{proposition}{Proposition}{Propositions}
\Crefname{proposition}{Proposition}{Propositions}
\crefname{corollary}{Corollary}{Corollaries}
\Crefname{corollary}{Corollary}{Corollaries}
\crefname{observation}{Observation}{Observations}
\Crefname{observation}{Observation}{Observations}
\crefname{definition}{Definition}{Definitions}
\Crefname{definition}{Definition}{Definitions}
\crefname{remark}{Remark}{Remarks}
\Crefname{remark}{Remark}{Remarks}
\crefname{section}{Section}{Sections}
\Crefname{section}{Section}{Sections}

\hypersetup{
  hidelinks,
  pdftitle={A counterexample to the Fang--Lin conjecture: Edge and spectral extremality diverge near a Turan graph},
  pdfauthor={Qi Wu and Yong Lu},
  pdfsubject={Extremal and spectral graph theory},
  pdfkeywords={spectral radius, non-r-partite graph, color-critical graph, Turan problem, Mycielskian}
}

\newtheorem{theorem}{Theorem}[section]

\newaliascnt{lemma}{theorem}
\newtheorem{lemma}[lemma]{Lemma}
\aliascntresetthe{lemma}

\newaliascnt{proposition}{theorem}
\newtheorem{proposition}[proposition]{Proposition}
\aliascntresetthe{proposition}

\newaliascnt{corollary}{theorem}
\newtheorem{corollary}[corollary]{Corollary}
\aliascntresetthe{corollary}

\theoremstyle{definition}
\newaliascnt{definition}{theorem}

\aliascntresetthe{definition}

\newaliascnt{remark}{theorem}

\aliascntresetthe{remark}

\newaliascnt{observation}{theorem}
\newtheorem{observation}[observation]{Observation}
\aliascntresetthe{observation}

\newcommand{\EX}{\operatorname{EX}}
\newcommand{\SPEX}{\operatorname{SPEX}}
\newcommand{\join}{\vee}

\title{A counterexample to the Fang--Lin conjecture: Edge and spectral extremality diverge near a Tur\'an graph\thanks{This work is supported by the National Natural Science Foundation of China (Nos.~12371348 and 12201258) and High-Quality Science and Technology Cultivation Project of Jiangsu Normal University (No.~JSNUGZL2026069).}}
\author{Qi Wu, Yong Lu\thanks{Corresponding author.}\\[2pt]
\small School of Mathematics and Statistics, Jiangsu Normal University,\\[-1pt]
\small Xuzhou, Jiangsu 221116, People's Republic of China\\[-1pt]
\small Emails:~\texttt{wuqimath@163.com, luyong@jsnu.edu.cn}}
\date{}

\begin{document}
\maketitle

\begin{abstract}
Fang and Lin [J. Algebraic Combin. 63 (2026), Art.~58] asked whether, whenever $F$ is edge-color-critical with $\chi(F)=r+1$, every non-$r$-partite, $F$-free graph of maximum adjacency spectral radius must also maximize the number of edges. We give a negative answer. Let $F=K_1\vee\mu(K_3)$, where $\mu(K_3)$ is the Mycielskian of a triangle. This graph is edge-color-critical with $\chi(F)=5$. We prove that $\SPEX_{5}(n,F)\cap\EX_{5}(n,F)=\varnothing$ for all sufficiently large $n$. Thus no graph can simultaneously maximize both the edge count and the spectral radius.
\medskip

\noindent\textbf{Keywords.} spectral radius; non-$r$-partite graph; color-critical graph; Tur\'an problem; Mycielskian.

\noindent\textbf{MSC 2020.} 05C35, 05C50.
\end{abstract}

\section{Introduction}

Tur\'an's theorem identifies the balanced complete $r$-partite graph $T_{n,r}$ as the unique $n$-vertex $K_{r+1}$-free graph with the largest number of edges \cite{Turan}. Throughout, we write $t_r(n):=e(T_{n,r})$. Simonovits \cite{Simonovits} extended this exact conclusion from cliques to every edge-color-critical graph $F$ with $\chi(F)=r+1$. The corresponding spectral statement for edge-color-critical graphs follows from Nikiforov's spectral saturation theorem \cite{Nikiforov2009}; the clique case is contained in his earlier spectral Tur\'an results \cite{Nikiforov2007}. This agreement between edge and spectral extremality is part of a broader phenomenon. For example, Wang et al.~\cite[Theorem~1.3]{WangKangXue} proved a general transfer theorem under the hypothesis $\operatorname{ex}(n,F)=t_r(n)+O(1)$.

A different problem appears when the host graph is required not to be $r$-partite. The Tur\'an graph is then excluded, and one asks for the first extremal layer outside the $r$-colorable class. For $K_{r+1}$, the edge problem goes back to Brouwer and was later revisited and sharpened in several forms \cite{Brouwer,KangPikhurko}. On the spectral side, Lin et al.~\cite{LinNingWu} solved the non-bipartite triangle-free case, and Li and Peng \cite{LiPeng} extended the result to non-$r$-partite $K_{r+1}$-free graphs. For books, the spectral-to-edge inclusion follows by combining the spectral classification of Liu and Miao \cite{LiuMiao} with the edge-extremal classification of Miao et al.~\cite{MiaoLiuVanDam}. Fang and Lin \cite{FangLin} proved the corresponding inclusion for $\theta(1,q,s)$ with $q,s\ge2$ and $q$ even, while Wang et al.~\cite{WangChenZhang} treated complete split graphs.

A closely related higher-chromatic layer has been studied for triangle-free graphs of chromatic number at least four. Ren et al.~\cite{RenWangWangYang} determined the edge-extremal graphs, and Zhu and Lin \cite{ZhuLin} determined the spectral-extremal graph for all sufficiently large orders. In that problem the two objectives select the same blow-up of the Gr\"otzsch graph. The example below shows that this agreement can fail for a different color-critical forbidden graph.

For the edge objective under an additional chromatic condition, Simonovits's general theory already gives finite symmetrized descriptions of sufficiently large extremal graphs and a linear first correction to the Tur\'an number \cite{SimonovitsAdditional}. Roberts and Scott proved a quantitative stability theorem for forbidden graphs with a critical edge \cite[Theorem~1.4]{RobertsScott}, and Hou et al.~proved the corresponding stability statement for suspensions of edge-critical graphs \cite[Theorem~1.5]{HouLiZeng}. These results place the edge-extremal problem near a complete multipartite graph, but they do not determine the present first layer or its spectral analogue.

For an $n$-vertex graph $G$, let $\rho(G)$ be the spectral radius of its adjacency matrix; standard background on adjacency spectra can be found in \cite{BrouwerHaemers}. A graph $F$ is called edge-color-critical if it has an edge whose deletion lowers its chromatic number. For such an $F$ with $\chi(F)=r+1$, define
\[
\operatorname{ex}_{r+1}(n,F)=\max\{e(G): |V(G)|=n,\ F\nsubseteq G,\ \chi(G)>r\},
\]
and let $\EX_{r+1}(n,F)$ be the family of graphs attaining this maximum. Define $\operatorname{spex}_{r+1}(n,F)$ and $\SPEX_{r+1}(n,F)$ analogously by maximizing the adjacency spectral radius. Motivated by the results above, Fang and Lin \cite{FangLin} proposed
\begin{equation}\label{eq:FL}
\SPEX_{r+1}(n,F)\subseteq\EX_{r+1}(n,F)
\end{equation}
for every edge-color-critical $F$ and all sufficiently large $n$. O and Wu recently proved \eqref{eq:FL} under an embeddability condition together with an exact hypothesis on the non-$r$-partite Tur\'an number, and they determined the spectral-extremal graphs for all edge-color-critical complete multipartite forbidden graphs \cite{OWu}. More precisely, their reduction theorem assumes $\operatorname{ex}_{r+1}(n,F)=t_r(n)-\lfloor n/r\rfloor+2(s-1)$ for an $s$-embeddable graph $F$; see \cite[Theorem~1.2]{OWu}. The example below has, for $r=4$, the different first correction $-n/2+O(1)$ and therefore lies outside the numerical range of that theorem. We nevertheless use several of their cleaning and one-vertex transfer results, always with explicit attribution.

We show that this conclusion fails without such a restriction. Let $\mu(K_3)$ denote the Mycielskian of a triangle and set $F=K_1\vee\mu(K_3)$. Mycielski introduced the operation $\mu$ in his construction of graphs with large chromatic number and small clique number \cite{Mycielski}. Lam et al.~proved that $\mu(K_3)$ is edge-color-critical with chromatic number $4$ \cite{LamLinGuSong}, and Hou et al.~showed that suspending an edge-critical graph preserves a critical edge \cite[Lemma~3.3]{HouLiZeng}. Thus $F$ is edge-color-critical with chromatic number $5$; we record this known fact in \Cref{prop:Fcritical}. Hou et al.~also determined the ordinary Tur\'an problem for suspensions of edge-critical graphs \cite[Theorem~1.4]{HouLiZeng}. Their result concerns the unrestricted $F$-free problem and does not determine the first extremal layer under the additional condition $\chi(G)>4$. The local coloring obstruction in the present layer can be arranged in two essentially different ways. This flexibility is exactly what separates the two extremal objectives.

Our main result is as follows.

\begin{theorem}\label{thm:main}
Let $F=K_1\vee\mu(K_3)$. For all sufficiently large $n$, we have $\SPEX_5(n,F)\cap\EX_5(n,F)=\varnothing$. In particular, the Fang--Lin conjecture is false. Moreover, $\operatorname{ex}_{5}(n,F)=t_4(n)-n/2+O(1)$, whereas every $G\in\SPEX_5(n,F)$ satisfies $e(G)=t_4(n)-3n/4+O(1)$ and $\rho(G)=3n/4-3/4+O(n^{-1})$.
\end{theorem}

The two constructions behind the theorem already reveal the direction of the separation.

\begin{theorem}\label{thm:pair}
Let $F=K_1\vee\mu(K_3)$. There exist $F$-free, non-$4$-partite graphs $A_n$ and $B_n$ of order $n$ such that $e(A_n)=t_4(n)-n/2+O(1)$, $e(B_n)=t_4(n)-3n/4+O(1)$, $\rho(A_n)=3n/4-5/6+O(n^{-1})$, and $\rho(B_n)=3n/4-3/4+O(n^{-1})$. Consequently, $e(A_n)>e(B_n)$ and $\rho(A_n)<\rho(B_n)$ for all sufficiently large $n$.
\end{theorem}

Thus the graph preferred by the spectral radius loses a linear number of edges. The source of the discrepancy is nevertheless finite. After removing a bounded exceptional set, every spectral-extremal graph, and a suitably symmetrized edge-extremal graph, has a balanced complete $4$-partite core. The exceptional vertices carry the obstruction to $4$-colorability. If $\delta_s$ measures the number of additional colors available to an exceptional vertex $s$, then the edge loss is linear in $\sum_s\delta_s$, whereas the first correction to the spectral radius is governed by
$\sum_s\frac{\delta_s(6-\delta_s)}{12}.$
The latter function is concave. The edge problem prefers two defects of size $1$; the spectral problem prefers one defect of size $3$.

The main technical issue is to show that no irregular graph outside this template family can improve either objective. On the edge side, Simonovits's theory under additional chromatic conditions already supplies a broad finite symmetrized description \cite{SimonovitsAdditional}. What is needed here is a sharper $F$-specific template that is shared by the edge and spectral objectives and preserves a fixed obstruction to $4$-colorability. We quote the near-Tur\'an cleaning and minimum-Perron-coordinate deletion inputs of O and Wu; their stability input is based on work of Desai et al.~\cite{DesaiEtAl,OWu}. We also quote their heavy-set Perron-coordinate estimate and prove only the alignment statement not covered by their formal hypotheses. The remaining bounded-certificate reduction uses classical replacement and symmetrization ideas. Its structural antecedents include Zykov's method \cite{Zykov}, strong Tur\'an stability and twin reductions \cite{TyomkynUzzell}, and recent bounded blow-up arguments in a related spectral problem \cite{ZhuLin}. Related bounded-template viewpoints also appear in Zhang's spectral-skeleton framework \cite{ZhangSkeletons} and in work on blowup thresholds \cite{HuangLiuRong}. None of these results directly supplies the certificate-preserving four-partite reduction needed here. We therefore prove that reduction, balance the four large classes up to a constant, and reduce the unrestricted problem to a finite template optimization.

The paper has five sections. After this introduction, Section~2 introduces the forbidden graph and the two competing constructions. Section~3 develops the bounded-template expansion and solves the resulting finite optimization problem. \Cref{sec:global-reduction} proves the global reduction from arbitrary extremal graphs to bounded uniform templates. Section~5 completes the proof of the main theorem and discusses the scope of the counterexample.

\section{The forbidden graph and the two competing constructions}

\subsection{The forbidden graph}

Let $M=\mu(K_3)$ be the Mycielskian of a triangle, following the construction introduced in \cite{Mycielski}. We use the vertex set
$V(M)=\{x_1,x_2,x_3,y_1,y_2,y_3,w\}.$
The vertices $x_1,x_2,x_3$ induce a triangle, $y_i$ is adjacent to $x_j$ exactly when $i\ne j$, and $w$ is adjacent to $y_1,y_2,y_3$. Let
$F=K_1\join M,$
and denote the additional complete vertex by $z$.

Lam et al.~proved that $M=\mu(K_3)$ is edge-color-critical with chromatic number $4$ \cite{LamLinGuSong}. Hou et al.~proved that if $uv$ is a critical edge of a graph $H$, then each of the two cone edges incident with $u$ and $v$ is critical in $K_1\vee H$ \cite[Lemma~3.3]{HouLiZeng}. Together with $\chi(K_1\vee H)=\chi(H)+1$, these cited results immediately give the following proposition.

\begin{proposition}\label{prop:Fcritical}
The graph $F$ is edge-color-critical and $\chi(F)=5$.
\end{proposition}

\paragraph{Neighborhood criterion.}
Since $F=K_1\join M$, every copy of $F$ has a vertex whose neighborhood contains a copy of $M$. Consequently, $G$ is $F$-free whenever every neighborhood $G[N_G(v)]$ is $M$-free; in particular, it is enough that every vertex neighborhood be $3$-colorable.

\subsection{The edge-favored construction}

Let $Q_1,Q_2,Q_3,Q_4$ be nonempty independent sets, and add all edges between distinct $Q_i$ and $Q_j$. Add four vertices $a,b,c,d$. Inside $\{a,b,c,d\}$ add the edges $ab$, $ac$, and $bd$.
Their adjacencies to the large classes are
\begin{align*}
a&\sim Q_3\cup Q_4,\\
b&\sim Q_2\cup Q_4,\\
c&\sim Q_1\cup Q_3\cup Q_4,\\
d&\sim Q_1\cup Q_2\cup Q_4.
\end{align*}
There are no other edges incident with $a,b,c,d$.

Write $A(Q_1,Q_2,Q_3,Q_4)$ for the resulting graph. For each $n$, choose the four class sizes to maximize the number of edges subject to $\sum_{i=1}^4|Q_i|=n-4$, and call the resulting graph $A_n$.

\begin{proposition}\label{prop:Achi}
The graph $A_n$ has chromatic number $5$.
\end{proposition}

\begin{proof}
In a proper $4$-coloring, each of the four nonempty classes $Q_1,Q_2,Q_3,Q_4$ must be monochromatic, and the four classes must receive distinct colors. Label these colors $1,2,3,4$, respectively.
The vertex $c$ is adjacent to $Q_1,Q_3,Q_4$, so it is forced to use color $2$. Similarly, $d$ is forced to use color $3$. Since $ac$ and $bd$ are edges, $a$ and $b$ are both forced to use color $1$, contradicting $ab\in E(A_n)$. Hence $\chi(A_n)>4$.

A proper $5$-coloring is obtained by coloring $Q_i$ with color $i$, taking $c=2$, $d=3$, $b=1$, and giving $a$ a fifth color. Thus $\chi(A_n)=5$.
\end{proof}

\begin{proposition}\label{prop:Afree}
The graph $A_n$ is $F$-free.
\end{proposition}

\begin{proof}
By the neighborhood criterion above, it is enough to prove that every vertex
neighborhood is $M$-free.

First consider a vertex outside $Q_4$.  The following displayed triples are
partitions of the indicated neighborhoods into independent sets:
\begin{align*}
N(v),\ v\in Q_1:&\quad Q_2\cup\{c\},\quad Q_3\cup\{d\},\quad Q_4,\\
N(v),\ v\in Q_2:&\quad Q_1\cup\{b\},\quad Q_3\cup\{d\},\quad Q_4,\\
N(v),\ v\in Q_3:&\quad Q_1\cup\{a\},\quad Q_2\cup\{c\},\quad Q_4.
\end{align*}
The neighborhoods of $a,b,c,d$ are respectively partitioned by
\begin{align*}
&Q_3\cup\{b\},\quad Q_4,\quad \{c\};\\
&Q_2\cup\{a\},\quad Q_4,\quad \{d\};\\
&Q_1\cup\{a\},\quad Q_3,\quad Q_4;\\
&Q_1\cup\{b\},\quad Q_2,\quad Q_4.
\end{align*}
Thus all these neighborhoods are $3$-colorable and hence $M$-free.

It remains to take $v\in Q_4$ and put $H=A_n[N(v)]$. The map interchanging $Q_2\leftrightarrow Q_3$, $a\leftrightarrow b$, and $c\leftrightarrow d$, and fixing $Q_1$, is an automorphism of $H$.  Hence, up to this
symmetry, every triangle of $H$ has one of the three types
$Q_1Q_2Q_3$, $Q_1Q_2d$ and $Q_2bd$.

Let $T=t_1t_2t_3$ be such a triangle.  If $T$ is the triangle
$x_1x_2x_3$ in a copy of $M$, then the shadow vertex corresponding to $t_i$ must lie in $C_i=(N_H(t_j)\cap N_H(t_k))\setminus V(T)$, where $\{i,j,k\}=\{1,2,3\}$, and the root must be adjacent to one vertex of each $C_i$.

For $T$ of type $Q_1Q_2Q_3$, label $t_i\in Q_i$ for
$i\in[3]$.  Then $C_1=Q_1\setminus\{t_1\}$, $C_2=(Q_2\setminus\{t_2\})\cup\{c\}$, and $C_3=(Q_3\setminus\{t_3\})\cup\{d\}$. The three sets $Q_1\cup\{a,b\}$, $Q_2\cup\{c\}$, and $Q_3\cup\{d\}$ partition $V(H)$, and each is anticomplete to the correspondingly
indexed set $C_i$.  Thus no root exists.

For $T$ of type $Q_1Q_2d$, label $t_1\in Q_1$, $t_2\in Q_2$ and
$t_3=d$.  Here $C_1=(Q_1\setminus\{t_1\})\cup\{b\}$, $C_2=Q_2\setminus\{t_2\}$, and $C_3=Q_3$. The partition into $Q_1\cup\{b\}$, $Q_2\cup\{a,c\}$, and $Q_3\cup\{d\}$ gives the same obstruction.  Finally, for $T$ of type $Q_2bd$, label
$t_1\in Q_2$, $t_2=b$ and $t_3=d$.  Then
$C_3=(N_H(t_1)\cap N_H(t_2))\setminus V(T)=\varnothing$, so even the
three shadow vertices cannot be chosen.  By symmetry this also covers
the types $Q_1Q_3c$ and $Q_3ac$.  Therefore $H$ is $M$-free, and hence
$A_n$ is $F$-free.
\end{proof}

\begin{proposition}\label{prop:Aedges}
The edge count of $A_n$ satisfies $e(A_n)=t_4(n)-n/2+O(1)$.
\end{proposition}

\begin{proof}
Let $q_i=|Q_i|$ and $N=n-4$. Directly from the definition,
\begin{align*}
e(A_n)
&=\sum_{1\le i<j\le4}q_iq_j
  +2(q_1+q_2+q_3)+4q_4+3\\
&=\frac12\left(N^2-\sum_{i=1}^4q_i^2\right)+2N+2q_4+3.
\end{align*}
The maximizing integer vector has
$q_i=\frac N4+O(1)$
for every $i$; more precisely, the continuous optimum has deviations
$(-\tfrac12,-\tfrac12,-\tfrac12,\tfrac32)$ from $N/4$.
Therefore
$e(A_n)=\frac{3N^2}{8}+\frac{5N}{2}+O(1)=\frac{3n^2}{8}-\frac n2+O(1).$
Since $t_4(n)=3n^2/8+O(1)$, the result follows.
\end{proof}

\subsection{The spectral-favored construction}

Let $P_1,P_2,P_3,P_4$ be nonempty independent sets, complete between distinct classes. Add vertices $c_1,c_2,c_3,c_4,w$. For $i\in[4]$, join $c_i$ to every vertex in $P_j$ if and only if $i\ne j$, and join $w$ to all four vertices $c_i$. There are no other edges involving the five new vertices. Denote the graph by
$B(P_1,P_2,P_3,P_4).$
For $B_n$, choose the $P_i$ as equally as possible subject to
$\sum_{i=1}^4|P_i|=n-5.$
Equivalently, $B_n$ is obtained from the Mycielskian $\mu(K_4)$ by
blowing up its four original vertices into the nonempty independent
sets $P_1,\ldots,P_4$. Mycielski's theorem gives $\chi(\mu(K_4))=5$
\cite{Mycielski}, and replacing a vertex by a nonempty independent class
of twins preserves the chromatic number. Hence we may record the following
direct consequence without a separate proof.

\begin{proposition}\label{prop:Bchi}
The graph $B_n$ has chromatic number $5$.
\end{proposition}

\begin{proposition}\label{prop:Bfree}
The graph $B_n$ is $F$-free.
\end{proposition}

\begin{proof}
Every vertex neighborhood is $3$-colorable. The neighborhood of $w$ is the independent set $\{c_1,c_2,c_3,c_4\}$. The neighborhood of $c_i$ consists of a complete $3$-partite graph on the classes $P_j$, $j\ne i$, together with the isolated vertex $w$. If $v\in P_i$, then $N(v)=\bigcup_{j\ne i}(P_j\cup\{c_j\})$, and the three sets $P_j\cup\{c_j\}$, $j\ne i$, are independent. The neighborhood criterion above now gives the result.
\end{proof}

\begin{proposition}\label{prop:Bedges}
The edge count of $B_n$ satisfies $e(B_n)=t_4(n)-3n/4+O(1)$.
\end{proposition}

\begin{proof}
Put $p_i=|P_i|$ and $N=n-5$. Then $e(B_n)=\sum_{i<j}p_ip_j+3N+4$.
The sum of cross edges is maximized by balancing the four classes, so $e(B_n)=3N^2/8+3N+O(1)=3n^2/8-3n/4+O(1)$.
\end{proof}

\section{Bounded templates and finite optimization}

\subsection{A spectral expansion for bounded uniform templates}

We record a general first-order expansion. Let $S$ be a fixed graph of order $s$. Let $Q_1,\ldots,Q_r$ be independent sets, complete between distinct classes. Each vertex $v\in S$ is assumed either complete or anticomplete to each $Q_i$. Let $R\in\{0,1\}^{s\times r}$ be the incidence matrix, where $R_{vi}=1$ exactly when $v$ is complete to $Q_i$, and put $k_v=\sum_{i=1}^rR_{vi}$.

The partition into $Q_1,\ldots,Q_r$ and the singleton vertices of $S$ is equitable. In the asymptotic regime below, the component containing the complete multipartite core has spectral radius of order $N$, whereas every component contained in $S$ has bounded spectral radius. Permuting vertices within any $Q_i$ is an automorphism, so averaging a nonnegative Perron vector over these permutations gives a Perron vector that is constant on every $Q_i$. Hence the spectral radius is the Perron root of the associated quotient matrix, in agreement with the standard theory of generalized joins and equitable quotients \cite{CardosoEtAl,YouYangSoXi}. What is needed below is a uniform first-order expansion of that bounded-dimensional Perron root.
The resulting quadratic coefficient is consistent with Zhang's walk-series expansion for graphs with a spanning complete multipartite core \cite{ZhangWalks} and with the one-exception estimate of O and Wu \cite[Lemma~4.4]{OWu}. We record and prove the required uniform expansion.

\begin{lemma}\label{lem:spectral-expansion}
Let $q_i=|Q_i|$, $N=\sum_iq_i$, and suppose $q_i=N/r+O(1)$ for every $i$. Then the resulting graph $G$ satisfies
\[
\rho(G)=\left(1-\frac1r\right)N
+\frac1{r(r-1)}\sum_{v\in S}k_v^2
+O(N^{-1}).
\]
Equivalently, since $n=N+s$,
\[
\rho(G)=\left(1-\frac1r\right)n
-\left(1-\frac1r\right)s
+\frac1{r(r-1)}\sum_{v\in S}k_v^2
+O(n^{-1}).
\]
For every fixed $C$, the error term is uniform over all templates with
$|S|\le C$ and $|q_i-N/r|\le C$.
\end{lemma}

\begin{proof}
The component containing the complete multipartite core has spectral radius
$\Theta(N)$, whereas every component contained entirely in $S$ has bounded
spectral radius. For large $N$, a nonnegative Perron vector for $\rho(G)$ is
therefore supported on the component containing the core. Averaging this
vector over all permutations within each $Q_i$ gives a Perron vector that is
constant on every $Q_i$, while the vertices of $S$ are singleton cells.
Hence $\rho(G)$ is the Perron root of the associated quotient matrix
\cite{YouYangSoXi}. We use the equivalent Perron equations below in order to
keep track of the error term uniformly.

Fix $C$ and assume throughout that $s\le C$ and $|q_i-N/r|\le C$. All implied constants below may depend on $C$ and $r$, but not on the particular template. Put $D=\operatorname{diag}(q_1,\ldots,q_r)$ and $\alpha=1-1/r$.

We first record the spectral radius of the complete multipartite core.
Let $\lambda_0=\rho(K_{q_1,\ldots,q_r})$.  If the value of a positive
Perron vector on its $i$th part is $u_i$ and $T=\sum_iq_i u_i$, then $(\lambda_0+q_i)u_i=T$.
Consequently $\lambda_0$ is the unique positive solution of
\begin{equation}\label{eq:core-secular}
\Phi(t):=\sum_{i=1}^r\frac{q_i}{t+q_i}=1.
\end{equation}
Let $q_i=N/r+d_i$, where $|d_i|\le C$ and $\sum_i d_i=0$.
Since the minimum and maximum degrees of the core are
$\alpha N+O_C(1)$, we have $\lambda_0=\alpha N+O_C(1)$.  Moreover, $\Phi(\alpha N)=\sum_{i=1}^r(N/r+d_i)/(N+d_i)=1+O_C(N^{-2})$, where the term of order $N^{-1}$ cancels because $\sum_i d_i=0$. Uniformly for $t=\alpha N+O_C(1)$, we have $\Phi'(t)=-\sum_{i=1}^r q_i/(t+q_i)^2=-1/N+O_C(N^{-2})$.
Applying the mean value theorem to \eqref{eq:core-secular}, we obtain
\begin{equation}\label{eq:core-radius}
\lambda_0=\alpha N+O_C(N^{-1}).
\end{equation}

The core is an induced subgraph of $G$, so \eqref{eq:core-radius} is a
lower bound for $\lambda:=\rho(G)$.  For the reverse rough bound, apply
the Collatz--Wielandt inequality to the positive vector which equals
$1$ on every $Q_i$ and $1/\alpha$ on $S$.  The quotient at a core
vertex is $N-q_i+O_C(1)=\alpha N+O_C(1)$, while at a vertex of $S$ it
is at most $\alpha N+O_C(1)$.  Hence
\begin{equation}\label{eq:lambda-rough}
\lambda=\alpha N+O_C(1).
\end{equation}

If $s=0$, the asserted expansion is exactly
\eqref{eq:core-radius}, so assume $s\ge1$.  Every component contained
entirely in $S$ has spectral radius at most $s-1$, whereas the component
containing the core has spectral radius at least $\lambda_0=\Theta(N)$.
Thus the Perron root is attained on the latter component.  A
nonnegative Perron vector may be chosen positive there and zero on any
other component.  By symmetry it is constant on each $Q_i$; write that
value as $y_i$, and write its restriction to $S$ as $z$.  Normalize it by $\sum_{i=1}^r q_i y_i=N$.
The Perron equations are
\begin{align}
(\lambda+q_i)y_i&=N+(R^{\mathsf T}z)_i,
   \label{eq:P1}\\
(\lambda I-A_S)z&=RDy.
   \label{eq:P2}
\end{align}
Let $Y=\max_i y_i$ and $Z=\max_{v\in S}z_v$. From \eqref{eq:P2} and \eqref{eq:lambda-rough}, $Z\le(sZ+NY)/\lambda=O_C(Y)$. From \eqref{eq:P1}, we obtain $Y=O_C(1)$, and hence $Z=O_C(1)$.
Since $\lambda+q_i=N+O_C(1)$, we obtain
\begin{equation}\label{eq:y-close}
y_i=1+O_C(N^{-1})
\qquad(i\in[r]).
\end{equation}

Let $k=R\mathbf 1$, whose entries are the numbers $k_v$.  By
\eqref{eq:y-close},
\begin{equation}\label{eq:RDy-expansion}
RDy=\frac Nr k+O_C(1).
\end{equation}
Since $\|A_S\|\le s-1=O_C(1)$ and $\lambda=\Theta(N)$, the Neumann
series gives, uniformly in the template,
\begin{equation}\label{eq:resolvent}
(\lambda I-A_S)^{-1}
 =\frac1\lambda I+O_C(N^{-2})
\end{equation}
in operator norm.  Combining \eqref{eq:P2},
\eqref{eq:RDy-expansion}, \eqref{eq:resolvent}, and
$\lambda=\alpha N+O_C(1)$, we obtain
\begin{equation}\label{eq:z-expansion}
z=\frac{N}{r\lambda}k+O_C(N^{-1})
 =\frac{k}{r-1}+O_C(N^{-1}).
\end{equation}
This also shows explicitly that the edges inside $S$ affect only the
$O_C(N^{-1})$ term.

Multiply the equation
$\lambda y_i=\sum_{j\ne i}q_jy_j+(R^{\mathsf T}z)_i$ by $q_i$ and sum
over $i$.  The normalization gives the exact identity
\begin{equation}\label{eq:lambda-identity}
\lambda
=N-\frac1N\sum_{i=1}^r q_i^2y_i
  +\frac1N z^{\mathsf T}Rq.
\end{equation}
Put $\varepsilon_i=y_i-1$. Since $\sum_iq_i\varepsilon_i=0$, we have $\sum_iq_i^2\varepsilon_i=\sum_i d_iq_i\varepsilon_i=O_C(1)$. Also $\sum_iq_i^2=N^2/r+O_C(1)$, so $\sum_iq_i^2y_i=N^2/r+O_C(1)$.
Moreover, by \eqref{eq:z-expansion}, we obtain
\[
z^{\mathsf T}Rq
=\left(\frac{k}{r-1}+O_C(N^{-1})\right)^{\!\mathsf T}
  \left(\frac Nr k+O_C(1)\right)
=\frac{N}{r(r-1)}\lVert k\rVert^2+O_C(1).
\]
Substituting into \eqref{eq:lambda-identity}, we obtain
\[
\lambda=\left(1-\frac1r\right)N
+\frac1{r(r-1)}\sum_{v\in S}k_v^2+O_C(N^{-1}).
\]
Replacing $N$ by $n-s$ gives the second form.  Every estimate was
uniform for $s\le C$, bounded class-size deviations, and all possible
adjacency patterns on the bounded set $S$, as required.
\end{proof}

Applying \Cref{lem:spectral-expansion} to the two explicit constructions, we obtain the following spectral radii.

\begin{corollary}\label{cor:spectraAB}
The graphs $A_n$ and $B_n$ satisfy $\rho(A_n)=3n/4-5/6+O(n^{-1})$ and $\rho(B_n)=3n/4-3/4+O(n^{-1})$.
\end{corollary}

\begin{proof}
For $A_n$, the exceptional set has order $4$ and its numbers of adjacent large classes are $2,2,3,3$.
By \Cref{lem:spectral-expansion}, $\rho(A_n)=3(n-4)/4+(2^2+2^2+3^2+3^2)/12+O(n^{-1})=3n/4-5/6+O(n^{-1})$. For $B_n$, the numbers are $3,3,3,3,0$, and hence $\rho(B_n)=3(n-5)/4+4\cdot3^2/12+O(n^{-1})=3n/4-3/4+O(n^{-1})$.
\end{proof}

We are now ready to prove \Cref{thm:pair}.

\begin{proof}[Proof of \Cref{thm:pair}]
Combine \Cref{prop:Bchi} with \Cref{prop:Achi,prop:Afree,prop:Aedges,prop:Bfree,prop:Bedges,cor:spectraAB}. In particular, $e(A_n)-e(B_n)=n/4+O(1)$ and $\rho(B_n)-\rho(A_n)=1/12+O(n^{-1})$, which are positive for all sufficiently large $n$.
\end{proof}

\subsection{The bounded-template optimization problem}

We now show that the preceding pair is not accidental. Fix a constant $C\ge 5$. Let $\mathfrak{B}_C(n)$ be the family of graphs $G$ admitting a partition $V(G)=S\mathbin{\dot\cup}Q_1\mathbin{\dot\cup}Q_2\mathbin{\dot\cup}Q_3\mathbin{\dot\cup}Q_4$ with the following properties:
\begin{enumerate}[label=(\roman*),leftmargin=2.2em]
\item $|S|\le C$;
\item each $Q_i$ is independent and all edges between distinct $Q_i,Q_j$ are present;
\item every $s\in S$ is either complete or anticomplete to each $Q_i$;
\item $\bigl||Q_i|-(n-|S|)/4\bigr|\le C$ for every $i$.
\end{enumerate}
Let $\mathfrak{B}_C(n,F)$ be the subfamily consisting of $F$-free, non-$4$-partite graphs.

For $s\in S$, put $I(s)=\{i:s\text{ is complete to }Q_i\}$ and $L(s)=[4]\setminus I(s)$.
Since the four nonempty classes are pairwise complete, the sets of colors used on distinct classes are disjoint in every proper $4$-coloring.  As there are four classes and only four colors, each $Q_i$ is monochromatic and the four classes receive pairwise distinct colors.  Thus $L(s)$ is exactly the list of colors available to $s$.

\begin{observation}\label{obs:list-obstruction}
For every sufficiently large $n$ and every $G\in\mathfrak{B}_C(n,F)$, the lists $L(s)$ are nonempty and the graph $G[S]$ is not colorable from these lists.
\end{observation}

\begin{proof}
If $L(s)=\varnothing$, then $s$ is adjacent to all four large classes. Fix a proper $4$-coloring of $M$. Since $\chi(M)=4$, all four colors occur, and for sufficiently large $n$ each color class can be embedded into the corresponding $Q_i$. The complete $4$-partite core then contains a copy of $M$, and together with $s$ it gives a copy of $F$. Thus all lists are nonempty.

A proper $4$-coloring of $G$ restricts to four distinct colors on $Q_1,\ldots,Q_4$ and then induces a proper list-coloring of $G[S]$. Conversely, every such list-coloring extends to $G$. Since $G$ is non-$4$-partite, no list-coloring exists.
\end{proof}

Define the defect of $s$ by $\delta_s=|L(s)|-1=3-|I(s)|\in\{0,1,2,3\}$, and put $D=\sum_{s\in S}\delta_s$ and $\Gamma=\sum_{s\in S}\delta_s(6-\delta_s)/12$.

\begin{lemma}\label{lem:defect}
For all sufficiently large $n$, if $G\in\mathfrak{B}_C(n,F)$, then $D\ge2$. If $D=2$, then there are exactly two vertices with defect $1$ and every other vertex has defect $0$.
\end{lemma}

\begin{proof}
All embeddings below are embeddings as ordinary subgraphs; additional edges among the selected vertices are harmless.

Assume first that $D=0$. By \Cref{obs:list-obstruction}, every list is nonempty, and hence every list is a singleton. Since the singleton assignment is not proper, there are adjacent vertices $x,y\in S$ with $L(x)=L(y)=\{1\}$.
Map the complete vertex $z$ to a vertex of $Q_2$. Map
$x_1,x_2,x_3$ to a vertex of $Q_3$, a vertex of $Q_4$, and $x$,
respectively. Map $y_1,y_2,y_3$ to another vertex of $Q_3$, the vertex
$y$, and a vertex of $Q_1$, respectively, and map the root $w$ to
another vertex of $Q_4$. All required edges are present because $x$ and
$y$ are complete to $Q_2,Q_3,Q_4$ and $xy$ is an edge. This gives a
copy of $F$, a contradiction.

Suppose $D=1$. If two adjacent vertices have the same singleton list, the previous argument applies. Otherwise there is one vertex $s$ with, after relabeling, $L(s)=\{1,2\}$,
and the singleton-list vertices are properly colored. Their forced
colors therefore give a proper coloring of $G[S\setminus\{s\}]$.
If some color in $L(s)$ were absent from $N_{G[S]}(s)$, assigning that
color to $s$ would extend this coloring to a list-coloring of $G[S]$.
Hence $s$ has a neighbor $x$ forced to color $1$ and a neighbor $y$
forced to color $2$.
Map $z$ to a vertex of $Q_3$, map $x_1,x_2,x_3$ to vertices of
$Q_4,Q_1,Q_2$, respectively, map $y_1,y_2,y_3$ to another vertex of
$Q_4$, the vertex $x$, and the vertex $y$, respectively, and map $w$ to
$s$. Again all required edges are present, producing $F$.

Finally suppose $D=2$ is concentrated at one vertex $s$, so $|L(s)|=3$. As before, we may assume that all singleton-list vertices are properly colored. Relabel so that $L(s)=\{1,2,3\}$.
As in the preceding case, the singleton lists give a proper coloring
of $G[S\setminus\{s\}]$. Every color in $L(s)$ must occur on a
neighbor of $s$, since otherwise that color could be assigned to $s$.
Thus there are neighbors $x_i$ of $s$ forced to colors $i=1,2,3$.
Map $z$ to a vertex of $Q_4$, map the triangle vertices of $M$ to
$Q_1,Q_2,Q_3$, respectively, map the corresponding shadow vertices to
$x_1,x_2,x_3$, and map $w$ to $s$. This is a copy of $F$, a
contradiction.

Thus $D\ge2$, and the only possible defect partition when $D=2$ is $1+1$.
\end{proof}

We next express the edge count in terms of the total defect.

\begin{lemma}\label{lem:edge-template}
Uniformly for $G\in\mathfrak{B}_C(n,F)$, we have $e(G)=t_4(n)-Dn/4+O_C(1)$.
\end{lemma}

\begin{proof}
Let $s=|S|$, $q_i=|Q_i|$, and let $k_v=|I(v)|=3-\delta_v$. Since $q_i=(n-s)/4+O_C(1)$,
\begin{align*}
e(G)
&=\sum_{i<j}q_iq_j+\sum_{v\in S}\sum_{i\in I(v)}q_i+e(G[S])\\
&=\frac38(n-s)^2+\frac{n-s}{4}\sum_{v\in S}k_v+O_C(1).
\end{align*}
Now $\sum_{v\in S}k_v=3s-D$.
Substitution and expansion give $e(G)=3n^2/8-Dn/4+O_C(1)=t_4(n)-Dn/4+O_C(1)$.
\end{proof}

Next we determine the spectral radius of a template in terms of the defect parameter $\Gamma$.

\begin{lemma}\label{lem:spectral-template}
Uniformly for $G\in\mathfrak{B}_C(n,F)$, we have $\rho(G)=3n/4-\Gamma+O_C(n^{-1})$.
\end{lemma}

\begin{proof}
By \Cref{lem:spectral-expansion}, $\rho(G)=3n/4-3|S|/4+(1/12)\sum_{v\in S}k_v^2+O_C(n^{-1})$.
Since $k_v=3-\delta_v$, we have $3/4-k_v^2/12=\delta_v(6-\delta_v)/12$. Summing over $S$, we obtain the formula.
\end{proof}

We are now ready to compare the two objectives inside $\mathfrak{B}_C(n,F)$.

\begin{theorem}\label{thm:template-separation}
Fix $C\ge5$. For all sufficiently large $n$, the edge-extremal and spectral-extremal families inside $\mathfrak{B}_C(n,F)$ are disjoint.
More precisely, every edge-extremal template has $D=2$, while every spectral-extremal template has $D=3$ concentrated at one vertex.
\end{theorem}

\begin{proof}
By \Cref{lem:defect,lem:edge-template}, the smallest possible linear edge defect is $D=2$, and it is achieved by $A_n$. Therefore every edge-extremal graph in the class has $D=2$ for sufficiently large $n$.

If $D=2$, by \Cref{lem:defect}, there are two defect-$1$ vertices, and hence $\Gamma=2\cdot1(6-1)/12=5/6$.
For $D\ge3$, each $\delta\in\{1,2,3\}$ satisfies $\delta(6-\delta)/12\ge\delta/4$.
Thus $\Gamma\ge D/4\ge3/4$, with equality exactly when $D=3$ is concentrated at one vertex of defect $3$. The graph $B_n$ realizes this equality. Every other admissible defect pattern has $\Gamma\ge5/6$: the pattern $(1,1)$ with $D=2$ gives $5/6$; the nonconcentrated patterns $(2,1)$ and $(1,1,1)$ with $D=3$ give $13/12$ and $5/4$, respectively; and $D\ge4$ gives $\Gamma\ge1$. Hence the gap above the minimum $3/4$, which is attained by the unique positive defect pattern $(3)$, is at least $1/12$. Since the error term in \Cref{lem:spectral-template} is uniform over $\mathfrak B_C(n,F)$, every spectral-extremal graph in the class has the concentrated defect-$3$ pattern for all sufficiently large $n$. Such a graph has $D=3$ and cannot be edge-extremal.
\end{proof}

\section{Global reduction to bounded uniform templates}\label{sec:global-reduction}

This section supplies the structural reduction from \Cref{thm:template-separation} to the unrestricted extremal problems. We separate quoted inputs, adapted arguments, and statements specific to the present forbidden graph. The cleaning package and the minimum-Perron-coordinate deletion statement are quoted from O and Wu \cite{OWu}; the heavy-set Perron-coordinate estimate is also quoted from them, while the alignment argument is included because their formal alignment lemma assumes stronger embeddability. The later classification by bounded neighborhood types and the Perron-vector cloning step have close precedents in strong Tur\'an stability and recent spectral blow-up reductions \cite{TyomkynUzzell,ZhuLin}. The bounded non-$4$-colorable certificate, the safe replacements preserving that certificate, and the resulting common four-partite reduction are specific to the present setting and are proved in full. Throughout, $F=K_1\vee\mu(K_3)$ and $f=|V(F)|=8$.
All constants depend only on $F$. We first choose the cleaning parameter $\eta>0$ sufficiently small that $\eta\le1/400$ and $C_1\eta\le1/2$, where $C_1$ is the absolute constant in the quoted Perron-coordinate estimate below, and then fix the corresponding stability parameter $\delta$ from \Cref{lem:cleaning}.  Next choose
$\theta>0$ sufficiently small for the bounded color certificate.  The
constants bounding $J$, and then $D_0$, $K$, and $B_0$, are fixed in
this order; finally $n$ is taken sufficiently large.  Thus no later
constant enters an earlier argument.

\subsection{Cleaning and deletion of one vertex}

The following package is the $r=4$ specialization of O--Wu's Theorem~2.2, Lemmas~2.9--2.14, and Remark~2.15 \cite{OWu}. After $\eta$ is fixed, their stability theorem supplies a fixed $\delta>0$. Remark~2.15 allows either the spectral hypothesis below or the edge lower bound $e(H)\ge t_4(n)-n$.

\begin{lemma}[O--Wu]\label{lem:cleaning}
There exists $\eta_0>0$ such that, for every fixed $0<\eta\le\eta_0$, there are constants $\delta=\delta(F,\eta)>0$ and $n_0=n_0(F,\eta)$ with the following property. Let $H$ be an $n$-vertex, $F$-free, non-$4$-partite graph, where $n\ge n_0$, and suppose that either $\rho(H)\ge(3/4-\delta)n$ or $e(H)\ge t_4(n)-n$. Then $H$ has a partition $V(H)=V_1\mathbin{\dot\cup}V_2\mathbin{\dot\cup}V_3\mathbin{\dot\cup}V_4$ maximizing the number of crossing edges, and a nonempty set $L\subseteq V(H)$ such that, with $\overline V_i=V_i\setminus L$,
\begin{enumerate}[label=(\roman*),leftmargin=2.2em]
\item $|L|\le\eta n$ and $\bigl||V_i|-n/4\bigr|\le\eta n$ for every $i$;
\item every $H[\overline V_i]$ is empty;
\item every $v\in L$ has $d_H(v)\le(3/4-12\eta)n$, whereas every $v\notin L$ has $d_H(v)>(3/4-12\eta)n$;
\item if $v\in\overline V_i$ and $j\ne i$, then $v$ has at most $21\eta n$ nonneighbors in $\overline V_j$;
\item if $i\in[4]$ and $X\subseteq\bigcup_{j\ne i}\overline V_j$ has $|X|\le f$, then $\left|\bigcap_{x\in X}N_{\overline V_i}(x)\right|>f$.
\end{enumerate}
\end{lemma}

\begin{lemma}\label{lem:edge-delete}
For every sufficiently large $n$ and every $G\in\EX_5(n,F)$, there exists $u\in V(G)$ such that $\chi(G-u)=4$.
\end{lemma}

\begin{proof}
The graph $A_n$ gives $e(G)\ge e(A_n)=t_4(n)-n/2+O(1)>t_4(n)-n$.
Apply the edge form of \Cref{lem:cleaning}, choose $u\in L$, and fix $i_0\in[4]$. Put $X=(\bigcup_{i\ne i_0}\overline V_i)\setminus\{u\}$.
For sufficiently small $\eta$ and large $n$, items (i) and (iii) give $|X|>d_G(u)$.
Form $G'$ by deleting all edges incident with $u$ and then joining $u$ to every vertex of $X$.

We claim that $G'$ is $F$-free. Otherwise a new copy $F'\cong F$ must contain $u$. Let $y_1,\ldots,y_q$ be the neighbors of $u$ in this copy. Then $q\le f-1$ and every $y_j$ lies in $X$. By \Cref{lem:cleaning}(v), there is a vertex $v_0\in\overline V_{i_0}\setminus V(F')$ adjacent to all $y_j$. Replacing $u$ by $v_0$ produces a copy of $F$ in $G$, a contradiction.

If $G-u$ were non-$4$-partite, then $G'$ would also be non-$4$-partite and $e(G')=e(G)-d_G(u)+|X|>e(G)$, contrary to edge extremality. Hence $\chi(G-u)\le4$. Since $G$ is non-$4$-partite and $\chi(G)\le\chi(G-u)+1$, we also have $\chi(G-u)\ge4$.
\end{proof}

By \Cref{cor:spectraAB}, every $G\in\SPEX_5(n,F)$ satisfies $\rho(G)\ge\rho(B_n)=3n/4-3/4+O(n^{-1})>3(n-5)/4$ for all sufficiently large $n$.  Since $F$ is connected and
edge-color-critical, the following is a direct instance of the transfer
principle of O and Wu.

\begin{lemma}[O--Wu, Remark~2.17]\label{lem:spectral-delete}
For every sufficiently large $n$ and every $G\in\SPEX_5(n,F)$, the graph $G$ is connected, and every vertex $u$ of minimum Perron coordinate satisfies $\chi(G-u)=4$.
\end{lemma}

\subsection{A four-partite core and a bounded color certificate}

Fix either an edge-extremal graph or a spectral-extremal graph $G$, and choose $u$ according to \Cref{lem:edge-delete,lem:spectral-delete}. We write $V(G-u)=P_1\mathbin{\dot\cup}P_2\mathbin{\dot\cup}P_3\mathbin{\dot\cup}P_4$, where every $P_i$ is independent.  For large $n$, the lower bound supplied by $A_n$ in the edge case
satisfies the edge hypothesis of \Cref{lem:cleaning}, while the lower bound supplied by
$B_n$ in the spectral case satisfies \Cref{lem:cleaning}.
Independently, fix a cleaning partition $V_1,\ldots,V_4$ and its exceptional set $L$ supplied by the applicable cleaning statement; in the spectral case, use the same graph $G$ and the minimum-Perron-coordinate vertex $u$ chosen above. Put $p_i=|P_i|$ and $M=\sum_{i<j}p_ip_j-e(G-u)$.

Roberts and Scott's stability theorem already gives quantitative near-$4$-partiteness from a near-Tur\'an edge bound \cite[Theorem~1.4]{RobertsScott}, and Hou et al.~proved the corresponding statement for suspensions \cite[Theorem~1.5]{HouLiZeng}. We need a bound relative to the fixed color classes of $G-u$, so we record the following short deduction from the edge--spectral Tur\'an theorem of Li et al.~\cite[Theorem~1.5]{LiLiuZhang}, Simonovits's color-critical Tur\'an theorem \cite{Simonovits}, and the variance identity.

\begin{lemma}[Linear crossing defect]\label{lem:linear-defect}
There is a constant $C_0$ such that $M\le C_0n$ and $p_i=n/4+O(\sqrt n)$ for every $i$.
\end{lemma}

\begin{proof}
For an edge-extremal graph, the lower bound supplied by $A_n$ and the inequality $d(u)\le n-1$ give $e(G-u)\ge t_4(n-1)-O(n)$.
For a spectral-extremal graph, first note that
$\rho(G)\ge\rho(B_n)=3n/4-O(1)$ and
$\rho(G)^2\le 2e(G)$, so $e(G)=\Omega(n^2)$.  Thus the
large-edge hypothesis in the edge-spectral Tur\'an theorem
of Li et al.~\cite[Theorem~1.5]{LiLiuZhang} is satisfied, and
that theorem gives $\rho(G)^2\le3e(G)/2$.
Together with Simonovits's theorem~\cite{Simonovits}, which gives $e(G)\le t_4(n)$ for
large $n$, it follows that $e(G)=t_4(n)-O(n)$ and hence again $e(G-u)\ge t_4(n-1)-O(n)$.
Since $e(G-u)=\sum_{i<j}p_ip_j-M\le t_4(n-1)-M$, we get $M=O(n)$. Finally, $t_4(n-1)-\sum_{i<j}p_ip_j=\frac12\sum_{i=1}^4(p_i-(n-1)/4)^2+O(1)$, which yields the claimed balance.
\end{proof}

For $v\in P_i$, define its full crossing defect by $\operatorname{def}(v)=|\{w\in V(G)\setminus P_i:vw\notin E(G)\}|$.
By \Cref{lem:linear-defect},
\begin{equation}\label{eq:sum-def}
\sum_{v\ne u}\operatorname{def}(v)\le 2M+n=O(n).
\end{equation}
Fix a constant $C_2>0$ such that the left-hand side of
\eqref{eq:sum-def} is at most $C_2n$ for all sufficiently large $n$.

\begin{lemma}[Bounded color certificate]\label{lem:certificate}
There is a constant $K_0$ and a set $J\subseteq V(G)$ such that $u\in J$, $|J|\le K_0$, and $\chi(G[J])>4$.
\end{lemma}

\begin{proof}
Choose a fixed $\theta>0$ with $\theta<1/100$, and put $B=\{v\ne u:\operatorname{def}(v)>\theta n\}$.
By \eqref{eq:sum-def}, $|B|\le C_2/\theta$. Let $U_i=P_i\setminus B$.
For all sufficiently large $n$, \Cref{lem:linear-defect} gives
$|U_i|\ge n/5$, and every vertex of $U_i$ has at most $\theta n$
nonneighbors outside $P_i$.

Choose $a_i\in U_i$ so that $a_1a_2a_3a_4$ is a $K_4$. At the $i$th
greedy step, the previously chosen vertices exclude at most
$3\theta n$ candidates from $U_i$, which is smaller than $|U_i|$.
Thus the choice is possible, and the four vertices are automatically
distinct because they lie in different parts.

For every $z\in B\cup\{u\}$ and every $i$ with $N(z)\cap U_i\ne\varnothing$, choose one vertex $w(z,i)\in N(z)\cap U_i$, and let $W$ be the set of distinct vertices selected in this way.
Since $|B|=O(1)$, we have $|W|=O(1)$.

For each $w\in W\cap(U_i\setminus\{a_i\})$ and each $j\ne i$, choose $y_j(w)\in U_j$ adjacent to $w$ and to every $a_k$ with $k\ne j$. Choose all these
auxiliary vertices pairwise distinct and outside
$\{a_1,a_2,a_3,a_4\}\cup W$.
At each step, the adjacency requirements exclude at most
$4\theta n$ vertices of $U_j$, while the requirement that all selected
vertices be distinct excludes only a bounded number of further
vertices. Since $|U_j|\ge n/5$ and $4\theta<1/5$, all choices are
possible for sufficiently large $n$.

In any proper $4$-coloring of the resulting gadget, after permuting
the colors so that $a_i$ has color $i$, every $y_j(w)$ is forced to
receive color $j$, and hence $w$ is forced to receive color $i$. The
base vertex $a_i$ is already forced to color $i$ and needs no
auxiliary gadget. If the same vertex belongs to $W$ for more than one
pair $(z,i)$, only this single forcing gadget is used.

Let $J$ consist of $B\cup\{u,a_1,a_2,a_3,a_4\}$ together with $W$ and
all auxiliary vertices. Its order is bounded.

If $G[J]$ had a proper $4$-coloring, color all of $U_i$ with color $i$ and retain the colors on $B\cup\{u\}$. If a vertex $z\in B\cup\{u\}$ has color $i$, then it has no neighbor in $U_i$, since otherwise the selected vertex $w(z,i)$ would be an adjacent vertex of the same forced color. As each $P_i$ is independent, this extends the coloring to all of $G$, a contradiction.
\end{proof}

O and Wu proved the following Perron-coordinate estimate in their Lemma~3.1(c) \cite{OWu}. Its proof uses only the cleaning conclusions and the minimum-Perron-coordinate deletion step, so the same $r=4$ statement applies in the present setting.

\begin{lemma}[O--Wu, Lemma~3.1(c)]\label{lem:heavy-coordinates}
Suppose that $G$ is spectral-extremal, let $\boldsymbol{x}$ be its positive Perron vector, put $\mu=\max_vx_v$, and define $H_i=V_i\setminus(L\cup\{u\})$ and $H=\bigcup_{i=1}^4H_i$. There is a constant $C_1>0$ such that, for sufficiently small fixed $\eta$ and all sufficiently large $n$, every $v\in H$ satisfies $(1-C_1\eta)\mu\le x_v\le\mu$.
\end{lemma}

The formal alignment lemma of O and Wu is stated under their stronger embeddability hypothesis. Since only the cleaning estimates and the $4$-coloring of $G-u$ are needed for the part used here, we give the short argument.

\begin{lemma}[Heavy-set alignment]\label{lem:heavy-alignment}
Suppose that $G$ is spectral-extremal and retain the notation of \Cref{lem:heavy-coordinates}. After relabeling the four color classes $P_1,\ldots,P_4$, we have $H_i\subseteq P_i$ for every $i\in[4]$. In particular, every $P_i$ contains a linear set of vertices whose Perron coordinates are at least $c_0\mu$ for some constant $c_0>0$.
\end{lemma}

\begin{proof}
By \Cref{lem:cleaning}(ii), each $H_i$ is independent. Items (i) and (iv) give $|H_i|\ge(1/4-2\eta)n-1$ and show that every vertex of $H_i$ has at most $21\eta n$ nonneighbors in $H_j$ whenever $i\ne j$.

Consider the fixed proper $4$-coloring $P_1\mathbin{\dot\cup}\cdots\mathbin{\dot\cup}P_4$ of $G-u$. For each $i$, some color class, say $P_{c_i}$, contains at least $|H_i|/4$ vertices of $H_i$. Our initial choice of the cleaning parameter ensures $\eta\le1/400$. Then $|H_i\cap P_{c_i}|>21\eta n$ for every sufficiently large $n$. The colors $c_1,\ldots,c_4$ are distinct: if $c_i=c_j$ for $i\ne j$, a vertex of $H_i\cap P_{c_i}$ is nonadjacent to every vertex of $H_j\cap P_{c_j}$, contrary to \Cref{lem:cleaning}(iv). Relabel so that $c_i=i$. If a vertex $v\in H_i$ belonged to $P_j$ with $j\ne i$, then $v$ would be nonadjacent to all vertices of $H_j\cap P_j$, giving the same contradiction. Hence $H_i\subseteq P_i$ for every $i$.

Each $H_i$ has linear order. Our initial choice of $\eta$ also ensures $C_1\eta\le1/2$, so \Cref{lem:heavy-coordinates} gives $x_v\ge\mu/2$ for every $v\in H$. The last assertion follows with $c_0=1/2$.
\end{proof}

By \Cref{lem:certificate}, every later modification that leaves the induced
subgraph $G[J]$ unchanged produces a graph of chromatic number at least five
and hence remains non-$4$-partite.

\subsection{Intersection-neighborhood replacements}

For all sufficiently large $n$, every $P_i\setminus J$ has at least
$n/5$ vertices. Choose an integer $D_0>10C_2$. By
\eqref{eq:sum-def}, fewer than $C_2n/D_0<n/10$ vertices of $G-u$ have
defect greater than $D_0$. Hence every $P_i\setminus J$ contains at
least $n/10$ vertices of defect at most $D_0$, and in particular at
least $f+1$ such vertices. Let $Z_i=\{z_{i,1},\ldots,z_{i,f+1}\}\subseteq P_i\setminus J$ with $\operatorname{def}(z_{i,t})\le D_0$, and set $R_i=\bigcap_{t=1}^{f+1}N_G(z_{i,t})$.
Then
\begin{equation}\label{eq:R-complement}
R_i\subseteq V(G)\setminus P_i,
\qquad
\bigl|(V(G)\setminus P_i)\setminus R_i\bigr|\le(f+1)D_0.
\end{equation}

\begin{lemma}[Safe intersection replacement]\label{lem:intersection-replacement}
Let $v\in P_i\setminus J$. Delete all edges incident with $v$ and set its new neighborhood equal to $R_i$. The resulting graph is $F$-free and non-$4$-partite.
\end{lemma}

\begin{proof}
The graph induced by $J$ is unchanged, so non-$4$-partiteness follows from \Cref{lem:certificate}. Suppose a new copy $F'\cong F$ appears. It must use $v$. Since $|Z_i|=f+1>|V(F')|$, some $z\in Z_i$ does not belong to $F'$. Every neighbor of $v$ in $F'$ lies in $R_i\subseteq N_G(z)$. Replacing $v$ by $z$ gives a copy of $F$ in the original graph, a contradiction.
\end{proof}

We now apply the safe replacement to bound the crossing defect outside $J$.

\begin{lemma}\label{lem:bounded-defect-outside-J}
There is a constant $K$ such that every $v\in V(G)\setminus J$ satisfies $\operatorname{def}(v)\le K$.
\end{lemma}

\begin{proof}
Suppose first that $G$ is edge-extremal. By \Cref{lem:intersection-replacement}, replacing $N(v)$ by $R_i$ is an admissible modification. Hence $d_G(v)\ge |R_i|$.
Both neighborhoods lie outside $P_i$, so \eqref{eq:R-complement} immediately gives $\operatorname{def}(v)=O(1)$.

Now suppose that $G$ is spectral-extremal, and let $\boldsymbol{x}$ be
its positive Perron vector with $\mu=\max_wx_w$.  By
\Cref{lem:heavy-alignment}, every $P_j$ contains a linear set $H_j$ on
which $x_w\ge c_0\mu$.  Since $R_i$ omits only $O(1)$ vertices outside
$P_i$, it contains all but $O(1)$ vertices of every $H_j$, $j\ne i$.
Thus
\begin{equation}\label{eq:Ri-heavy}
\sum_{w\in R_i}x_w\ge c_1n\mu
\end{equation}
for a constant $c_1>0$.

The replacement in \Cref{lem:intersection-replacement} is admissible.
Testing its adjacency matrix with the Perron vector of $G$ shows that
spectral extremality forces $\sum_{w\in N(v)}x_w\ge\sum_{w\in R_i}x_w$. Together with \eqref{eq:Ri-heavy} and $\rho(G)\le n-1$, the Perron
equation gives $x_v\ge c_1\mu$ for every $v\notin J$, after decreasing
$c_1$ if necessary.  Furthermore, put $T_i=\sum_{w\notin P_i}x_w$.
Since $P_i$ is independent and both $N(v)$ and $R_i$ lie outside
$P_i$, the preceding weight comparison gives
\begin{align*}
\sum_{\substack{w\notin P_i\\vw\notin E(G)}}x_w
&=T_i-\sum_{w\in N(v)}x_w\\
&\le T_i-\sum_{w\in R_i}x_w\\
&=\sum_{w\in(V(G)\setminus P_i)\setminus R_i}x_w
 =O(\mu)
\end{align*}
by \eqref{eq:R-complement}.  Every missing vertex outside $J$ has
coordinate at least $c_1\mu$, and $J$ has bounded order.  Therefore
$\operatorname{def}(v)=O(1)$.
\end{proof}

\subsection{Large types and twin classes}

The use of bounded neighborhood types and twin classes is classical in
strong stability and symmetrization arguments \cite{SimonovitsAdditional,TyomkynUzzell}. A closely related spectral proof for triangle-free graphs of chromatic number at least four was given by Zhu and Lin \cite{ZhuLin}. We need a certificate-preserving version, because every modification below must preserve both $F$-freeness and the fixed obstruction $G[J]$ to $4$-colorability.

Classify vertices of $P_i\setminus J$ by their neighborhood in $J$: $C(i,A)=\{v\in P_i\setminus J:N(v)\cap J=A\}$ for $A\subseteq J$.
There are only $O(1)$ such types. Choose a constant $B_0>(f+1)K+2f$, and call a type large if its order exceeds $B_0$.

\begin{lemma}\label{lem:large-cross-complete}
Any two large types belonging to different parts are complete to each other.
\end{lemma}

\begin{proof}
Suppose $X\subseteq P_i$ and $Y\subseteq P_j$, $i\ne j$, are large types and $xy$ is a missing edge with $x\in X$, $y\in Y$. Adding $xy$ leaves $G[J]$ unchanged and hence preserves non-$4$-partiteness. In the edge case it increases the number of edges, while in the spectral case it strictly increases the spectral radius because $G$ is connected and its Perron vector is positive. Consequently $G+xy$ cannot remain $F$-free, and it contains a copy $F'\cong F$ using the new edge.

Choose $x'\in X\setminus V(F')$ adjacent to every vertex of
$F'\setminus\{x,y\}$ that must be adjacent to $x$.  Conditions involving vertices of $J$ are automatic because $x'$ has the same $J$-type as $x$.  Every required neighbor of $x$ outside $J$ lies outside $P_i$, since $P_i$ is independent and all edges incident with $x$ other than the new edge $xy$ already belong to $G$.  Therefore \Cref{lem:bounded-defect-outside-J} applies to each such requirement. Excluding $V(F')$ removes at most $f$ candidates, and each of the at most $f-2$ required neighbors outside $J$ removes at most $K$ candidates. Hence at most
$f+(f-2)K<B_0$ candidates are excluded.

Next choose $y'\in Y\setminus(V(F')\cup\{x'\})$ satisfying all adjacency requirements of $y$ and also $x'y'\in E(G)$.  Every required neighbor of $y$ inherited from $F'$ and lying outside $J$ is outside $P_j$, because $P_j$ is independent; hence the same defect bound applies.  The forbidden set removes at most $f+1$ candidates, the required neighbors inherited from $y$ remove at most $(f-2)K$ candidates, and the condition $x'y'\in E(G)$ removes at most another $K$ candidates.  Thus at most
$f+1+(f-1)K<B_0$ candidates are excluded.  Replacing $x,y$ by $x',y'$ in
$F'$ produces a copy of $F$ in $G$, a contradiction.
\end{proof}

\begin{lemma}[Bounded twin refinement]\label{lem:twin-refinement}
There is a bounded set $S\supseteq J$ such that each $P_i\setminus S$ is
nonempty, every retained refined class of $P_i\setminus S$ is a twin
class, and any two surviving classes in different parts are complete to
each other. Moreover every vertex of $S$ is either complete or
anticomplete to each surviving class.
\end{lemma}

\begin{proof}
There are at most $4\cdot2^{|J|}$ initial types.  Put every initial type
of order at most $B_0$ into $S_0$, together with $J$; hence $|S_0|$ is
bounded.  Refine each remaining initial type by the full neighborhood
vector on $S_0$.  There are at most $2^{|S_0|}$ refined classes inside
each initial type.  Move every refined class of order at most $B_0$ into
$S$ and retain all other refined classes.

The total order of $S$ is bounded. The total number of refined classes
is bounded as well. If some $P_i$ contained no retained refined class,
then every refined class in $P_i$ would have order at most $B_0$, and
hence $|P_i|=O(1)$. This contradicts
$|P_i|=n/4+O(\sqrt n)$ from \Cref{lem:linear-defect}. Thus each
$P_i\setminus S$ is nonempty.

A retained class is uniform on $S_0$ by its definition. Two retained
classes in the same part are anticomplete, and two retained classes in
different parts are complete by \Cref{lem:large-cross-complete}, because
their initial types were large. Finally, a vertex moved at the second
refinement stage belongs to a large initial type; it is anticomplete to
retained classes in its own part and complete to retained classes in
other parts. Thus every vertex of $S$ is uniform on every retained
class, and the retained classes are genuine twin classes.
\end{proof}

The next argument is a constrained form of Zykov symmetrization
\cite{Zykov}, within the broader structural tradition of Simonovits's
extremal theory under additional chromatic conditions
\cite{SimonovitsAdditional}.  The cloning operations here must preserve
not only the objective but also $F$-freeness and the fixed
non-$4$-colorable certificate $G[J]$, so we keep the argument in full.

\begin{lemma}\label{lem:one-class-per-part}
After an objective-preserving modification in the edge-extremal case, every part $P_i\setminus S$ consists of one twin class. In the spectral-extremal case this already holds in $G$.
\end{lemma}

\begin{proof}
Let $A,B\subseteq P_i\setminus S$ be distinct surviving twin classes.
They both have more than $f$ vertices, and they differ only in their
neighborhoods in $S$.

For a twin class $T$, write $N(T)=N_G(t)$ for any $t\in T$. In the
spectral case, permutations within $T$ are automorphisms of the
connected graph $G$, so the positive Perron vector is constant on $T$.
Let $a$ and $b$ be its common coordinates on $A$ and $B$, respectively,
and put $\sigma_A=\sum_{w\in N(A)}x_w=\rho a$ and $\sigma_B=\sum_{w\in N(B)}x_w=\rho b$. Choose $v\in A$ and fix $b_0\in B$. Form $G'$ by leaving all adjacencies not incident with $v$ unchanged and setting $N_{G'}(v)=N_G(b_0)=N(B)$.
No loop is created because $A\cup B\subseteq P_i$ is independent. The
modified graph is $F$-free: if a new copy uses $v$, an unused vertex of
$B$ replaces it. It also preserves $G[J]$. Let $\boldsymbol{x}$ be the
old Perron vector. Only the edges incident with $v$ change, so
\[
\boldsymbol{x}^{\mathsf T}A(G')\boldsymbol{x}
 -\boldsymbol{x}^{\mathsf T}A(G)\boldsymbol{x}
 =2a(\sigma_B-\sigma_A).
\]
Consequently,
\[
\frac{\boldsymbol{x}^{\mathsf T}A(G')\boldsymbol{x}}
 {\|\boldsymbol{x}\|_2^2}
 =\rho(G)+\frac{2a(\sigma_B-\sigma_A)}{\|\boldsymbol{x}\|_2^2}.
\]
Since $G'$ is admissible, spectral extremality gives $\rho(G')\le\rho(G)$, whereas the Rayleigh principle gives the displayed quotient at most $\rho(G')$.  Thus $\sigma_B\le\sigma_A$. Reversing $A$ and $B$ gives the opposite inequality, so $\sigma_A=\sigma_B$ and, from $\rho a=\sigma_A=\sigma_B=\rho b$, also $a=b$.

For the first replacement the displayed quotient is therefore exactly
$\rho(G)$. Hence $\rho(G')=\rho(G)$. Equality in the Rayleigh principle
for the real symmetric matrix $A(G')$ implies that $\boldsymbol{x}$
belongs to its top eigenspace, so $\boldsymbol{x}$ is an eigenvector for
$\rho(G')$. All adjacencies outside $S$ coincide for vertices in the same part. Since
$A$ and $B$ are distinct, some $s\in S$ distinguishes them; the right-hand
side of the eigen-equation at $s$ changes by $a$ or by $-a$. Since $a>0$,
the old vector cannot remain an eigenvector, a contradiction.  Hence only one class can
remain in each part.

Now assume that $G$ is edge-extremal. Let $d_A$ and $d_B$ denote the
common degrees of the vertices in $A$ and $B$, respectively. Replacing one vertex of $A$ by a
clone of $B$ changes the edge count by $d_B-d_A$ and is admissible by
the same unused-clone argument.  The reverse replacement is also
admissible, so $d_A=d_B$.  Since $A\cup B\subseteq P_i$ is independent,
cloning one vertex of $A$ to the type of $B$ changes no edge incident
with any other vertex of $A\cup B$.  Consequently the degrees of all
remaining $A$-vertices and all $B$-type vertices are still $d_A=d_B$.
Inductively, replacing the vertices of $A$ one at a time by clones of
$B$ preserves the edge count at every step.  The unused-clone argument
continues to preserve $F$-freeness, and $G[J]$ is unchanged throughout.
The two classes therefore merge without losing extremality or
non-$4$-partiteness.  Repeating this process leaves one twin class in
each part.
\end{proof}

Thus there is a partition $V(G)=S\mathbin{\dot\cup}Q_1\mathbin{\dot\cup}Q_2\mathbin{\dot\cup}Q_3\mathbin{\dot\cup}Q_4$, where $|S|=O(1)$, each $Q_i$ is independent, distinct $Q_i,Q_j$ are complete to each other, and every $s\in S$ is complete or anticomplete to each $Q_i$.

\subsection{Constant-order balance}

\begin{lemma}\label{lem:constant-balance}
For both extremal problems, $\bigl||Q_i|-(n-|S|)/4\bigr|=O(1)$ for every $i\in[4]$.
\end{lemma}

\begin{proof}
Let $q_i=|Q_i|$. Since $S$ is bounded and the original color classes have order $n/4+O(\sqrt n)$, every $q_i$ is linear in $n$.

Suppose first that $G$ is spectral-extremal, and let $\xi_i$ be the
common Perron coordinate on $Q_i$. By \Cref{cor:spectraAB},
$\rho(G)\ge\rho(B_n)=3n/4-O(1)$, while $\rho(G)\le n-1$; hence
$\rho(G)=\Theta(n)$.
Fix distinct $i,j$, choose $x\in Q_i$ and $y\in Q_j$, and form $G'$ by
leaving all adjacencies not incident with $x$ unchanged and setting $N_{G'}(x)=N_G(y)\setminus\{x\}$.
Equivalently, update the core partition to
$Q_i'=Q_i\setminus\{x\}$ and $Q_j'=Q_j\cup\{x\}$; then $x$ has the
same type as the vertices of $Q_j$. Since $|Q_j|$ is linear in $n$ and
$f$ is fixed, any new copy of $F$ using $x$ leaves an unused vertex of
$Q_j$ that can replace $x$. Thus the operation is $F$-free and
preserves $G[J]$.

Evaluate the Rayleigh quotient of $G'$ at the Perron vector of $G$. The old
and new neighborhood weights of $x$ are $\rho\xi_i$ and
$\rho\xi_j-\xi_i$, respectively. Hence the change in the adjacency quadratic form at this vector is $2\xi_i(\rho\xi_j-(\rho+1)\xi_i)$.
The modified graph is admissible, so spectral extremality and the Rayleigh principle imply that this quantity is nonpositive. Since $\xi_i>0$, we have $(\rho+1)\xi_i\ge\rho\xi_j$.
The reverse move gives the symmetric inequality, and therefore, with
$\xi_*:=\max_k\xi_k$,
\begin{equation}\label{eq:xi-close}
|\xi_i-\xi_j|=O(\xi_*/n).
\end{equation}
In particular every $\xi_i=(1+O(n^{-1}))\xi_*$.

Let $z_s$ be the Perron coordinate of $s\in S$ and put $b_i=\sum_{s\in S,\,s\sim Q_i}z_s$.
If $z_*:=\max_{s\in S}z_s$, the Perron equation at a vertex attaining $z_*$ gives $(\rho-|S|)z_*\le n\xi_*$,
so $z_*=O(\xi_*)$ and hence $b_i=O(\xi_*)$. The equation on $Q_i$ is $\rho\xi_i=\sum_{k\ne i}q_k\xi_k+b_i$.
Subtracting the equations for $i$ and $j$ and using \eqref{eq:xi-close} gives $q_i\xi_i-q_j\xi_j=O(\xi_*)$.
Since $q_j=O(n)$ and $|\xi_i-\xi_j|=O(\xi_*/n)$, it follows that
$|q_i-q_j|\xi_i=O(\xi_*)$.  The coordinates $\xi_i$ are comparable to
$\xi_*$, so $|q_i-q_j|=O(1)$.

For an edge-extremal template, use the same move and let $a_i=|\{s\in S:s\text{ is complete to }Q_i\}|$.
Moving one vertex from the type of $Q_i$ to the type of $Q_j$ changes the edge count by $q_i-q_j-1+a_j-a_i$. Both directions are admissible, so $|(q_i-q_j)-(a_i-a_j)|\le1$.
Since $|a_i-a_j|\le|S|$, we again obtain $|q_i-q_j|=O(1)$.
The asserted balance follows by summing the four class sizes.
\end{proof}

Simonovits's theory under additional chromatic conditions already gives, for the edge objective, a finite symmetrized description of sufficiently large extremal graphs \cite{SimonovitsAdditional}. Tyomkyn and Uzzell obtained a strong twin-structure form of Tur\'an stability \cite{TyomkynUzzell}, while spectral skeletons, exact blow-up structures and a related higher-chromatic spectral extremal problem have been studied in other settings \cite{ZhangSkeletons,HuangLiuRong,ZhuLin}. The theorem below gives the sharper form needed here: it produces the same certificate-preserving four-partite template for every spectral-extremal graph and for an edge-extremal representative.

\begin{theorem}[Global bounded-template reduction]\label{thm:global-reduction}
There is a constant $C$ such that, for all sufficiently large $n$,
\begin{enumerate}[label=(\roman*),leftmargin=2.2em]
\item every graph in $\SPEX_5(n,F)$ belongs to $\mathfrak B_C(n,F)$;
\item at least one graph in $\EX_5(n,F)$ belongs to $\mathfrak B_C(n,F)$.
\end{enumerate}
\end{theorem}

\begin{proof}
Let $G$ be extremal for one of the two objectives.  By
\Cref{lem:edge-delete,lem:spectral-delete}, there is a vertex $u$ such
that $G-u$ has independent color classes $P_1,\ldots,P_4$.
\Cref{lem:linear-defect} gives a linear bound on the total number of
missing crossing edges and also
$|P_i|=n/4+O(\sqrt n)$.  The construction in
\Cref{lem:certificate} then produces a set $J$ of uniformly bounded
order with $\chi(G[J])>4$.  In particular, every later operation that
keeps the induced subgraph $G[J]$ unchanged preserves non-$4$-partiteness.

The safe replacements of \Cref{lem:intersection-replacement}, used as
comparison operations rather than performed on the spectral extremal
graph, imply through \Cref{lem:bounded-defect-outside-J} that every
vertex outside $J$ has crossing defect at most a constant $K$.
Consequently there are only boundedly many $J$-types. After all small
types and all small refined classes are moved into the exceptional set,
\Cref{lem:twin-refinement} yields a set $S\supseteq J$ of uniformly
bounded order such that every $P_i\setminus S$ is nonempty, the
surviving classes are twin classes, classes in the same $P_i$ are
anticomplete, classes in different parts are complete, and each vertex
of $S$ is uniform on every surviving class.

By \Cref{lem:one-class-per-part}, in the spectral case there is already
only one surviving class in each part, so no graph modification is
needed.  In the edge case, the successive mergers used there preserve the
number of vertices, the number of edges, $F$-freeness, and the induced
subgraph on $J$.  They therefore preserve both edge extremality and
non-$4$-partiteness, and produce an edge-extremal graph with one
surviving class $Q_i$ in each part; in this case we rename the resulting
graph as $G$.  Thus in either case we have $V(G)=S\mathbin{\dot\cup}Q_1\mathbin{\dot\cup}Q_2\mathbin{\dot\cup}Q_3\mathbin{\dot\cup}Q_4$, where the $Q_i$ are independent, distinct $Q_i,Q_j$ are complete to
one another, and every vertex of $S$ is complete or anticomplete to
each $Q_i$.  Finally, \Cref{lem:constant-balance} gives $\bigl||Q_i|-(n-|S|)/4\bigr|=O(1)$ for every $i\in[4]$.
All bounds depend only on $F$.  Choose one constant $C\ge5$ larger than
the bound on $|S|$ and all four balance constants.  The spectral graph
itself, and the edge-extremal graph obtained by the admissible mergers,
then belong to $\mathfrak B_C(n,F)$, proving (i) and (ii).
\end{proof}

\section{Proof of the main theorem and concluding remarks}

\subsection{Proof of the main theorem}

\begin{proof}[Proof of \Cref{thm:main}]
The criticality and chromatic number of $F$ are recorded in \Cref{prop:Fcritical} from the cited results of Lam et al.~and Hou et al. Choose $C\ge5$ large enough for \Cref{thm:global-reduction}. By \Cref{thm:template-separation}, the edge-extremal graphs in $\mathfrak B_C(n,F)$ have $D=2$, while the spectral-extremal graphs in this class have $D=3$ concentrated at one vertex.

Every global spectral-extremal graph belongs to $\mathfrak B_C(n,F)$. If such a graph were also globally edge extremal, it would be both edge extremal and spectral extremal inside $\mathfrak B_C(n,F)$, contradicting \Cref{thm:template-separation}. Hence $\SPEX_5(n,F)\cap\EX_5(n,F)=\varnothing$.

By \Cref{thm:global-reduction}(ii), a global edge-extremal graph may be taken in $\mathfrak B_C(n,F)$; \Cref{lem:edge-template,thm:template-separation} then give $\operatorname{ex}_5(n,F)=t_4(n)-n/2+O(1)$.
Likewise, \Cref{lem:edge-template,lem:spectral-template,thm:template-separation} give, for every global spectral-extremal graph,
\[
e(G)=t_4(n)-\frac{3n}{4}+O(1),
\qquad
\rho(G)=\frac{3n}{4}-\frac34+O(n^{-1}).
\]
\end{proof}

\subsection{Concluding remarks}

The example identifies a genuine boundary for spectral-to-edge transfer
in non-\(r\)-partite Tur\'an problems. In the positive results known so
far, the departure from the Tur\'an graph is usually governed by one
local obstruction, or by several obstructions whose costs are forced to
agree. For \(F = K_1 \vee \mu(K_3)\) there are two inequivalent ways to
realize the necessary failure of \(4\)-colorability. The edge count
rewards the cheaper total defect, while the spectral radius rewards a
more concentrated defect. The two objectives therefore select different
global extremal graphs.

The bounded-template reduction is also useful beyond the particular
counterexample: it shows that a near-Tur\'an spectral-extremal graph can
sometimes be reduced to a finite coloring problem even when the extremal
correction is linear in \(n\). The ingredients specific to the present
problem are the \(F\)-specific defect parameter, the bounded
color-forcing certificate, the common certificate-preserving template
for the two objectives, and the finite optimization that separates edge
and spectral extremality.

\section*{Data availability}

No data were used for the research described in this article.

\section*{Declaration on the use of generative AI}

The authors used ChatGPT (OpenAI) to discuss proof strategies, check intermediate arguments, and improve the exposition. The authors reviewed and verified the mathematical content and take full responsibility for the manuscript.

\end{document}